\documentclass[10pt]{article}

\usepackage[english]{babel}
\usepackage[tbtags]{amsmath}
\usepackage{amsfonts,amssymb,amsthm,cite}

\usepackage{color}

\usepackage[width=120mm, left=15mm, top=10mm, height=180mm, paper=a5paper]{geometry}

\usepackage{sectsty}

\sectionfont{\fontsize{10}{13}\selectfont}

\newtheorem{thm}{Theorem}[section]
\newtheorem{lem}[thm]{Lemma}
\newtheorem{cor}[thm]{Corollary}

\def\esup{\mathop{\rm ess\,sup}}

\def\loc{\text{\rm loc}}
\def\Re{\mathbb{R}}
\def\Na{\mathbb{N}}
\def\leb{{\cal L}^1}
\def\supp{\mathop{\rm supp}}
\def\strong{\mathrm{s}}
\def\weak{\mathrm{w}}

\def\Ch#1{\mathrm{Ch}_{{#1},w}(I)}
\def\Cs#1{\mathrm{Cs}_{{#1},w}(I)}
\def\Chx#1{\mathrm{Ch}_{{#1},w_0}(I_0)}
\def\Csx#1{\mathrm{Cs}_{{#1},w_0}(I_0)}

\def\pio#1{\pi(#1)}

\def\ef#1{\mathbf{#1}}

\begin{document}

\begin{center}
 \sc\Large On the associated spaces for the weighted  altered Ces\`aro space $\Ch{\infty}$
\end{center}

\begin{center}
Dmitrii V. Prokhorov\\
Computing Center of the Far Eastern Branch\\
of the Russian Academy of Sciences\\
680000 Russia, Khabarovsk, Kim Yu Chena 65\\
prohorov@as.khb.ru
\end{center}

\noindent{\bf Abstract:} We study weigted altered Ces\`aro space $\Ch{\infty}$, which is non-ideal enlargement of the usual Ces\`aro space. We prove the connection of the space with one weighted Sobolev space of first order on real line and give characterization of associate spaces for the space.
\vskip 2mm

\noindent{\bf AMS 2020 Mathematics Subject Classification:} 46E30.
\vskip 2mm

\noindent{\bf Key words and phrases:}  {Ces\`aro function spaces, associated spaces.}

\section{Introduction}

Let  $I:=(c,d)\subset\mathbb{R}$, $p\in [1,\infty]$,  $p':=\frac{p}{p-1}$, ${\cal L}^1$ be the Lebesgue measure on $I$, $\mathfrak{M}(I)$ be the vector space of all ${\cal L}^1$-measurable functions $f:I\to [-\infty,\infty]$, $N:=\{f\in \mathfrak{M}(I):f=0\ {\cal L}^1\text{-a.e. on}\ I\}$,  $\pi$ be the  quotient map of $\mathfrak{M}(I)$ onto $L^0(I):=\mathfrak{M}(I)/N$, $L^p(I)$ be the  Lebesgue space.   We write elements of the space $L^0(I)$ by bold letters: 
$\ef{f}$, $\ef{g}$, and we denote by $\pio{f}$ the element of $L^0(I)$, which contain  $f\in \mathfrak{M}(I)$.   Also we put 
$$L^p_{\loc}(I):=\Big\{\ef{f}\in L^0(I):\|\chi_{(a,b)}\ef{f}\|_{L^p(I)}<\infty,\,\forall a,b\in I\Big\},$$
$$L^p_{\loc}([c,d)):=\Big\{\ef{f}\in L^0(I):\|\chi_{(c,x)}\ef{f}\|_{L^p(I)}<\infty,\,\forall x\in I\Big\},$$
$$L^p_{\loc}((c,d]):=\Big\{\ef{f}\in L^0(I):\|\chi_{(x,d)}\ef{f}\|_{L^p(I)}<\infty,\,\forall x\in I\Big\}.$$
Let  
\begin{equation}\label{genw}
w\in \mathfrak{M}(I),\  w>0\ \leb\text{-a.e. on }I,\ \pio{w}\in L^p_{\loc}((c,d]) 
\end{equation}
and
$$\rho(f):=\begin{cases}\left(\int_I \left|w(x)\int_c^x f\right|^p\,dx\right)^\frac{1}{p}, & p\in [1,\infty),\\
                        \esup_{x\in I}\left|w(x)\int_c^x f\right|, & p=\infty;
                       \end{cases}$$
$$\Cs{p}:=\Big\{\ef{f}\in L^1_\loc([c,d))\,\Big|\,\|\ef{f}\|_{\Cs{p}}<\infty\Big\},\ \|\ef{f}\|_{\Cs{p}}:=\rho(|f|),$$
$$\Ch{p}:=\Big\{\ef{f}\in L^1_\loc([c,d))\,\Big|\,\|\ef{f}\|_{\Ch{p}}<\infty\Big\},\ \|\ef{f}\|_{\Ch{p}}:=\rho(f),\ f\in\ef{f}.$$
Clear  $\Cs{p}\subset \Ch{p}$. Since $w\in\eqref{genw}$  then $\ef{f}\in L^0(I)$ with compact support belongs to the space $\Cs{p}$.
The space $(\Cs{p},\|\cdot\|_{\Cs{p}})$ is named  weighted Ces\`aro space, it has been actively studied (see \cite{LM,KamKub} and the survey \cite{AM}). The space $(\Ch{p},\|\cdot\|_{\Ch{p}})$ we call weighted altered  Ces\`aro space.  It seems that there were no works about the spaces $\Ch{p}$ until recently  \cite{Pr2022,Pr2022-1,Step2022}.
 
The classic Ces\`aro space $\Csx{p}$ (where $I_0:=(0,\infty)$ and $w_0(x):=\frac{1}{x}$, $x\in I_0$) are studing since 1970s.  For $p\in (1,\infty)$ both spaces $\Csx{p}$ and $\Chx{p}$ are appeared during solving of problem of describing of accosiated spaces with one weighted Sobolev space on real line \cite{PSU0,PSU-UMN}, defined 
\begin{equation}\label{opredW}
W^1_{p}(I_0):=\{\ef{f}\in L^1_\loc(I_0): D\ef{f}\in L^1_\loc(I_0)\,\&\,\|\ef{f}\|_{W^1_{p}(I_0)}<\infty\},
\end{equation}
where $\|\ef{f}\|_{W^1_{p}(I_0)}:=\|\ef{f}\|_{L^{p}(I_0)}+\|{\textstyle \frac{1}{w_0}} D\ef{f}\|_{L^{p}(I_0)}$. As proved in \cite[Theorem 3.3]{Pr2022}
$$(W_p^1(I_0))'_\strong=\Csx{p'},\ \ \|\cdot\|_{(W_p^1(I_0))'_\strong}\approx\|\cdot\|_{\Csx{p'}};$$
$$(W_p^1(I_0))'_\weak=\Csx{p'},\ \ \|\cdot\|_{(W_p^1(I_0))'_\weak}\approx\|\cdot\|_{\Chx{p'}},$$
and the definition of associated spaces is given below. 

Let $(X,\|\cdot\|)$ be the normed space of elements of $L^0(I)$. The space associated (K\"othe dual space) with $X$ is
$$X'_\strong:=(X,\|\cdot\|)'_\strong:=\left\{\ef{g}\in L^0(I)\,\Big|\,\|\ef{g}\|_{X'_\strong}:=\sup_{\ef{h}\in X\setminus\{0\}}\frac{\|\ef{h}\ef{g}\|_{L^1(I)}}{\|\ef{h}\|}<\infty\right\}.$$
The ``strong'' space $X'_\strong$ does not always match the ``weak'' space 
\begin{align*}
&X'_\weak:=(X,\|\cdot\|)'_\weak:=\\
&=\left\{\ef{g}\in L^0(I)\,\Big|\,\ef{f}\ef{g}\in L^1(I)\,\forall\ef{f}\in X\ \&\ \|\ef{g}\|_{X'_\weak}:=\!\!\!\sup_{h\in\ef{h}\in X\setminus\{0\}}\!\!\frac{|\int_I hg|}{\|\ef{h}\|}<\infty,g\in\ef{g}\right\}, 
\end{align*}
which is  isomorphic to the subspace of the set  $X^*$ of all continuous functionals of the form $\ef{f}\mapsto \int_I fg$, $f\in \ef{f}\in X$, where $g\in\ef{g}$. Clear $X'_\strong\subset X'_\weak$.

From the definition of associated spaces $(\Cs{p})'_\strong=(\Cs{p})'_\weak$ and $\|\ef{g}\|_{(\Cs{p})'_\strong}=\|\ef{g}\|_{(\Cs{p})'_\weak}$ for $\ef{g}\in(\Cs{p})'_\strong$ and $p\in [1,\infty]$.
For $p\in [1,\infty)$ the space $\Cs{p}$ is  order ideal and it has absolutely continuous norm. Then   for  $\Lambda\in (\Cs{p})^*$ there exists $\ef{g}\in (\Cs{p})'_\strong$ such that $\|\Lambda\|_{(\Cs{p})^*}=\|\ef{g}\|_{(\Cs{p})'_\strong}$ and $\Lambda \ef{f}=\int_I fg$, $f\in\ef{f}\in\Cs{p}$, $g\in\ef{g}$ (see \cite[Chapter 1, Theorem 4.1]{BenSharp-1988}).

A problem of describing of associated   with $\Cs{p}$ spaces was solved in \cite{KamKub} with help of essential $\int_x^d w^p$-concave majorant  (see \cite[Definition 2.11]{KamKub}), and in  \cite{StepMZ2022} with help of  monotone majorant.

For $p\in [1,\infty)$ characterization of dual spaces for  weighted altered Ces\`aro space are given in the following theorem (see \cite{Pr2022-1}).

\begin{thm}\label{Ath}
Let  $p\in [1,\infty)$, $w\in \eqref{genw}$, $\ef{g}\in L^0(I)$. 

$1$. The equality $(\Ch{p})'_\strong=\{0\}$ holds.

$2$. The following statements are equivalent:
 
  $(i)$ $\ef{g}\in (\Cs{p},\|\cdot\|_{\Ch{p}})'_\weak;\vphantom{\Big[}$ 
  
  $(ii)$ $\ef{g}\cap AC_\loc(I)\not=\emptyset$, and for $\tilde g=\ef{g}\cap AC_\loc(I)$ there are $\tilde g(x)=-\int_x^d D\tilde g$, $x\in I$, and $\pio{\frac{D\tilde g}{w}}\in L^{p'}(I)$.
  
\  And $\|\ef{g}\|_{(\Cs{p},\|\cdot\|_{\Ch{p}})'_\weak}=\|\pio{\frac{D\tilde g}{w}}\|_{L^{p'}(I)}$.

$3$.   The following statements are equivalent:

  $(i)$ $\ef{g}\in (\Ch{p})'_\weak;\vphantom{\Big[}$ 
  
  $(ii)$ $\ef{g}\in L^\infty(I)$, $\ef{g}\cap AC_\loc(I)\not=\emptyset$, and for  $\tilde g=\ef{g}\cap AC_\loc(I)$ there are $\chi_{(b,d)}\tilde g=0$ for some  $b\in I$, and $\pio{\frac{D\tilde g}{w}}\in L^{p'}(I)$.
   
\  And $\|\ef{g}\|_{(\Ch{p})'_\weak}=\|\pio{\frac{D\tilde g}{w}}\|_{L^{p'}(I)}$.

$4$. The following statements are equivalent:
 
  $(i)$ $\Lambda\in (\Ch{p})^*;\vphantom{\Big[}$ 
  
  $(ii)$ there exists $h\in\mathfrak{M}(I)$ such that $\pio{\frac{h}{w}}\in L^{p'}(I)$,
\begin{equation*}%\label{predstfunk}
\Lambda \ef{f}=\int_I h(x)\left[\int_c^x f\right]\,dx,\ \ f\in\ef{f}\in\Ch{p}. 
\end{equation*}

\ And $\|\Lambda\|_{(\Ch{p})^*}=\|\pio{\frac{h}{w}}\|_{L^{p'}(I)}$.
\end{thm}

The key step of the proof of Theorem \ref{Ath} was the approximation  of an element of the space  $\Ch{p}$   by elements with compact support. For $p=\infty$ there is no such approximation but it is possible (see Section \ref{main}) to   describe of the associated spaces with $\Ch{\infty}$ for  the weight 
\begin{equation}\label{weights}\begin{split}
w(x)=\left[\int_c^x v\right]^{-1}\in (0,\infty),\  x\in I,\  v\in \mathfrak{M}(I),\\ \pio{v}\in L^1_\loc([c,d)),\ \lim_{b\to d-}\int_c^b v=\infty.\end{split}
\end{equation}

Throughout this paper $A\lesssim B$ and $B\gtrsim A$ means that $A\le cB$, where the constant $c$ depends only on $p$  and may be different
in different places. If both $A\lesssim B$ and $A\gtrsim B$, then
we write $A\approx B$. $\mathbb{N}$ stands for the set of positive integers, $\mathbb{R}$ is the set of all real numbers, $Df$ is the derivative of function $f:I\to\Re$, $D\ef{f}$ is the weak derivative of $\ef{f}\in L^0(I)$.  $AC(I)$ is the space of all  absolutely continuous functions $f:I\to \mathbb{R}$, the space of all locally absolutely continuous functions  is denoted by $AC_\loc(I)$. The symbol $BPV(I)$ designates the space of all functions $f:I\to\Re$ having a bounded pointwise variation (see \cite[2.1]{Leoni}). For any Borel measure $\lambda$, defined on Borel subsets of $I$, 
the symbol $\|\lambda\|$ means $|\lambda|(I)$, where  
$|\lambda|$  is  the total variation of $\lambda$. If $f\in BPV(I)$ then $\lambda_f$  designates the unique real Borel measure such that $\lambda_f((a,b])=f(b+)-f(a+)$ for all $a,b\in I$, with $a\le b$ (see \cite[Theorem 5.13]{Leoni}).  
$C_c^1(I)$ is the space of all real continuously differentiable  functions  with the compact support in $I$;
$C_0(I)$ is the space of all real continuous functions on $I$ which vanish at infinity (see \cite[3.16]{Rudin}).

\section{Connection with Sobolev space}

Let $p\in(1,\infty)$, $I_0:=(0,\infty)$, $w_0(x):=\frac{1}{x}$, $x\in I_0$,
$$Y:=\{f\in AC(I_0)\,|\, \supp f\text{ is a compact in }I_0\}$$
and $W_p^1(I_0)$ is defined by \eqref{opredW}. 
By using \cite[Remark 5.1]{PSU-UMN}, 
$$\ef{g}\in (\pio{Y},\|\cdot\|_{W_p^1(I_0)})'_\weak\ \ \Leftrightarrow\ \  \ef{g}\in L^1_\loc(I_0)\ \&\  [\mathbb{G}(\ef{g})+\mathfrak{G}(\ef{g})]<\infty,$$
where for $g\in\ef{g}\in L^1_\loc(I_0)$ 
$$\mathbb{G}(\ef{g})\approx\left(\int_0^\infty \frac{1}{t^{p'}}\left|\int_\frac{t}{2}^t g\right|^{p'}\,dt\right)^\frac{1}{p'}$$
and
\begin{align*}
\mathfrak{G}(\ef{g})&\approx \left(\int_0^\infty \frac{1}{t^{p'(2-p')}}\left|\int_t^{2t}y^{-p'}\left[\int_\frac{y}{2}^t g\right] \,dy\right|^{p'}\,dt\right)^\frac{1}{p'}\\
&=\left(\int_0^\infty \frac{1}{t^{p'(2-p')}}\left|\int_\frac{t}{2}^t g(x)\left[\int_t^{2x} y^{-p'}\,dy\right] \,dx\right|^{p'}\,dt\right)^\frac{1}{p'}\\
&=\left(\int_0^\infty \frac{1}{t^{p'(2-p')}}\left|\int_\frac{t}{2}^t g(x)\left[\frac{t^{1-p'}-(2x)^{1-p'}}{p'-1}\right] \,dx\right|^{p'}\,dt\right)^\frac{1}{p'}\\
&=\frac{1}{p'-1}\left(\int_0^\infty \frac{1}{t^{p'}}\left|\int_\frac{t}{2}^t g(x)\left[1-\left(\frac{2x}{t}\right)^{1-p'}\right] \,dx\right|^{p'}\,dt\right)^\frac{1}{p'}\\
&=\frac{1}{p'-1}\left(\int_0^\infty \frac{1}{t^{p'}}\left|\int_\frac{t}{2}^t g-\left(\frac{2}{t}\right)^{1-p'}\int_\frac{t}{2}^t g(x)x^{1-p'}\,dx\right|^{p'}\,dt\right)^\frac{1}{p'}. 
\end{align*}
Hence, for $\ef{g}\in L^1_\loc(I_0)$  the inequality   $[\mathbb{G}(\ef{g})+\mathfrak{G}(\ef{g})]<\infty$ is equivalent to
$$\left(\int_0^\infty \frac{1}{t^{p'}}\left|\int_\frac{t}{2}^t g\right|^{p'}\,dt\right)^\frac{1}{p'}+\left(\int_0^\infty \frac{1}{t^{p'(2-p')}}\left|\int_\frac{t}{2}^t g(x)x^{1-p'}\,dx\right|^{p'}\,dt\right)^\frac{1}{p'}<\infty.$$

{\bf Example.} Let $p=p'=2$,  $g(x):=\frac{1}{x}\sin\frac{1}{x}\chi_{(0,1]}(x)$, $x\in I_0$, and $\ef{g}\ni g$. Then
$$\int_0^1|g(x)|\,dx=\int_0^1\frac{|\sin\frac{1}{x}|}{x}\,dx=\int_1^\infty\frac{|\sin y|}{y}\,dy=\infty,$$
$$\int_1^\infty \left|\int_\frac{t}{2}^t \frac{g(x)}{x}\,dx\right|^2\,dt=\int_1^2 \left|\int_\frac{t}{2}^1 \frac{\sin \frac{1}{x}}{x^2}\,dx\right|^2\,dt\le 4,$$
\begin{align*}
\int_0^1 \left|\int_\frac{t}{2}^t \frac{g(x)}{x}\,dx\right|^2\,dt&=\int_0^1 \left|\int_\frac{t}{2}^t \frac{\sin \frac{1}{x}}{x^2}\,dx\right|^2\,dt=\int_0^1 \left|\int_\frac{1}{t}^\frac{2}{t} \sin y\,dy\right|^2\,dt\\
&=\int_1^\infty \frac{1}{x^2}\left|\int_x^{2x}\sin y\,dy\right|^2\,dx<\infty, 
\end{align*}
$$\int_1^\infty \frac{1}{t^2}\left|\int_\frac{t}{2}^t g\right|^2\,dt=\int_1^2 \frac{1}{t^2}\left|\int_\frac{t}{2}^1 \frac{\sin \frac{1}{x}}{x}\,dx\right|^2\,dt\le 1.$$
Besides that, from
$$\left|\int_y^{2y}\frac{\sin t}{t}\,dt\right|=\left|\int_y^{2y}\frac{d\cos t}{t}\right|=\left|\frac{\cos t}{t}\Big|_y^{2y}+\int_y^{2y}\frac{\cos t}{t^2}\,dt\right|\le \frac{5}{2y},$$
we have the estimate
\begin{align*}
\int_0^1 \frac{1}{t^2}\left|\int_\frac{t}{2}^t g\right|^2\,dt&=\int_0^1\frac{1}{t^2}\left|\int_\frac{t}{2}^t\frac{\sin\frac{1}{x}}{x}\,dx\right|^2\,dt=\int_0^1\frac{1}{t^2}\left|\int_\frac{1}{t}^\frac{2}{t}\frac{\sin y}{y}\,dy\right|^2\,dt\\
&=\int_1^\infty\left|\int_x^{2x}\frac{\sin y}{y}\,dy\right|^2\,dx\le\frac{25}{4}\int_1^\infty\frac{dy}{y^2}<\infty.
\end{align*}
Therefore, $\ef{g}\in L^1_\loc(I_0)\setminus L^1_\loc([0,\infty))$ and  $[\mathbb{G}(\ef{g})+\mathfrak{G}(\ef{g})]<\infty$, that is 
$$\ef{g}\in (\pio{Y},\|\cdot\|_{W_2^1(I_0)})'_\weak\setminus \Chx{2}.$$

\begin{thm} Let 
$p\in(1,\infty)$, $I_0:=(0,\infty)$, $w_0(x):=\frac{1}{x}$, $x\in I_0$,
$$X:=\{f\in AC(I_0)\,|\, \exists f(0+),\ \exists b\in I_0:\,\chi_{(b,\infty)}f=0\}$$
and $W_p^1(I_0)$ is defined by \eqref{opredW}.  Then 
$$(\pio{X},\|\cdot\|_{W_p^1(I_0)})'_\weak=\Chx{p'},\ \ \|\cdot\|_{(\pio{X},\|\cdot\|_{W_p^1(I_0)})'_\weak}\approx\|\cdot\|_{\Chx{p'}}.$$
\end{thm}

\begin{proof}
 For  $g\in \ef{g}\in L^1_\loc([0,\infty))$ we have (see the proof of \cite[Theorem 3.3]{Pr2022}) 
\begin{align}
&[\mathbb{G}(\ef{g})+\mathfrak{G}(\ef{g})]\lesssim \left(\int_0^\infty \frac{1}{t^{p'}}\left|\int_0^t g\right|^{p'}\,dt\right)^\frac{1}{p'}=\left(\int_0^\infty \frac{1}{t^{p'}}\left|\sum_{j\ge 0}\int_{t 2^{-(j+1)}}^{t2^{-j}} g\right|^{p'}\,dt\right)^\frac{1}{p'}\label{forloc}\\
&\le \sum_{j\ge 0}\left(\int_0^\infty \frac{1}{t^{p'}}\left|\int_{t 2^{-(j+1)}}^{t2^{-j}} g\right|^{p'}\,dt\right)^\frac{1}{p'}=\sum_{j\ge 0}(2^j)^{-\frac{1}{p}}\left(\int_0^\infty \frac{1}{x^{p'}}\left|\int_\frac{x}{2}^x g\right|^{p'}\,dx\right)^\frac{1}{p'}\nonumber\\
&=\frac{2^\frac{1}{p}}{2^\frac{1}{p}-1}\left(\int_0^\infty \frac{1}{x^{p'}}\left|\int_\frac{x}{2}^x g\right|^{p'}\,dx\right)^\frac{1}{p'}\approx \mathbb{G}(\ef{g}).\nonumber
\end{align}

Let $\ef{g}\in (\pio{X},\|\cdot\|_{W_p^1(I_0)})'_\weak$ and $g\in\ef{g}$. Since $Y\subset X$ then   $\ef{g}\in L^1_\loc(I_0)$ and
$$\left(\int_0^\infty \frac{1}{t^{p'}}\left|\int_\frac{t}{2}^t g\right|^{p'}\,dt\right)^\frac{1}{p'}<\infty.$$
For any  $b\in I_0$ a function $f\in C^\infty(I)$ with properties: $f=1$ on $(0,b)$ and  $f=0$ on $(b+1,\infty)$ belongs to the space $X$. It implies $\ef{g}\in L^1_\loc([0,\infty))$ and by \eqref{forloc}
$\ef{g}\in \Chx{p'}$ and $\|\ef{g}\|_{(\pio{X},\|\cdot\|_{W_p^1(I_0)})'_\weak}\approx\|\ef{g}\|_{\Chx{p'}}$.

Now let  $\ef{g}\in \Chx{p'}$. Then $\ef{g}\in L^1_\loc([0,\infty))$ and by \eqref{forloc} for any  $f\in Y$ the inequality
\begin{equation} \label{intocenka}
\left|\int_{I_0} fg\right|\lesssim \|\ef{g}\|_{\Chx{p'}}\|\pio{f}\|_{W_p^1(I_0)}.
\end{equation}
holds. Fix an arbitrary $f\in X$. Let $b\in I_0$ such that $f\chi_{(b,\infty)}=0$. For $0<\alpha_1<\alpha_2<\beta_2<\beta_1<\infty$ we put
\begin{equation*}%\label{opredomega}
  \omega_{\alpha_1,\alpha_2,\beta_1,\beta_2}(x):=
  \begin{cases}
  \left(\int_{\alpha_1}^{\alpha_2}\frac{dt}{t^{p'}}\right)^{-1}\int_{\alpha_1}^x\frac{dt}{t^{p'}},& x\in (\alpha_1,\alpha_2),\\
  1,& x\in [\alpha_2,\beta_2],\\
  \left(\int_{\beta_2}^{\beta_1}\frac{dt}{t^{p'}}\right)^{-1}\int_x^{\beta_1}\frac{dt}{t^{p'}},& x\in (\beta_2,\beta_1),\\
  0,& x\in I_0\setminus(\alpha_1,\beta_1).
  \end{cases}
 \end{equation*}
 Then  $\omega_{\alpha_1,\alpha_2,\beta_1,\beta_2} f\in Y$ and $\lim_{\genfrac{}{}{0pt}{2}{\alpha_2\to 0+}{\beta_2\to\infty}}\|\pio{f-\omega_{\alpha_1,\alpha_2,\beta_1,\beta_2} f}\|_{W_p^1(I_0)}=0$.
Also,
$$\int_{I_0}|\omega_{\alpha_1,\alpha_2,\beta_1,\beta_2} fg|\le \max_{x\in (0,b)}|f(x)|\int_0^b|g|<\infty,$$
and, by passing to the limit as $\alpha_2\to 0+$, $\beta_2\to\infty$, by Lebesgue's dominant convergence theorem we get  the estimate  \eqref{intocenka} for $f\in X$. 
\end{proof}

 Now we  show that in case decreasing weight $w\in \eqref{weights}$ the spaces $\Cs{\infty}$ and $\Ch{\infty}$  are associated with some  weighted Sobolev space of the first order on $I$.
 
 \begin{thm}
Let $w\in \eqref{weights}$, $v>0$ $\leb$-a.e. on $I$, 
$$X:=\{f\in AC(I),\,|\, \exists f(c+),\ \exists b\in I:\,\chi_{(b,d)}f=0\}$$
and
 \begin{equation*}
W^1_1(I):=\{\ef{f}\in L^1_\loc(I): D\ef{f}\in L^1_\loc(I)\,\&\,\|\ef{f}\|_{W^1_1(I)}<\infty\},
\end{equation*}
where $\|\ef{f}\|_{W^1_1(I)}:=\|v\ef{f}\|_{L^1(I)}+\|{\textstyle \frac{1}{w}} D\ef{f}\|_{L^1(I)}$.
Then 
\begin{align}
(W_1^1(I))'_\strong&=\Cs{\infty},\ \ \|\cdot\|_{(W_1^1(I))'_\strong}\approx\|\cdot\|_{\Cs{\infty}};\label{Wstr}\\
(W_1^1(I))'_\weak&=\Cs{\infty},\ \ \|\cdot\|_{(W_1^1(I))'_\weak}\approx\|\cdot\|_{\Ch{\infty}};\label{Wslab}\\
(\pio{X},\|\cdot\|_{W_1^1(I)})'_\weak&=\Ch{\infty},\ \ \|\cdot\|_{(\pio{X},\|\cdot\|_{W_1^1(I)})'_\weak}\approx\|\cdot\|_{\Ch{\infty}}.\label{WX}
\end{align}
 \end{thm}

 \begin{proof}
Fix any  $\ef{f}\in W^1_1(I)$. There exists $f=\ef{f}\cap AC_\loc(I)$ and for any $x,y\in I$ such that $x>y$ we have
$$|f(x)-f(y)|\le \int_y^x|Df|\le \|\chi_{(y,d)}\pio{w}\|_{L^{\infty}(I)}\|\chi_{(y,d)}{\textstyle \frac{1}{w}} D\ef{f}\|_{L^1(I)}.$$ 
 Hence, there exists the limit $f(d-)$. Since $\|v\ef{f}\|_{L^1(I)}<\infty$ and $v\not\in L^1_\loc((c,d])$ then $f(d-)=0$. Also $\pio{w}\in L^{\infty}_\loc((c,d])$ implies $D\ef{f}\in L^{1}_\loc((c,d])$. Consequently, $f(x)=-\int_x^d Df$, $x\in I$.
 
Also for an arbitrary $h\in\ef{h}\in L^1(I)$ since $\pio{w}\in L^{\infty}_\loc((c,d])$ then  $w\ef{h}\in L^1_\loc((c,d])$ and for $f_h(y):=\int_y^d wh$, $y\in I$, $\ef{f_h}\ni f_h$  we have  
$$\|v\ef{f_h}\|_{L^1(I)}=\int_I \left|v(y)\int_y^d wh\right|\,dy\le\|\ef{h}\|_{L^1(I)},\ \  \
 \|{\textstyle \frac{1}{w}} D\ef{f_h}\|_{L^1(I)}=\|\ef{h}\|_{L^1(I)},$$ 
 that is $\ef{f_h}\in W_1^1(I)$ and $\|\ef{f_h}\|_{W_1^1(I)}\le 2\|\ef{h}\|_{L^1(I)}$.
 
If $g\in\ef{g}\in \Cs{\infty}$ then for any $\ef{f}\in W^1_1(I)\setminus\{0\}$
\begin{equation}\label{upperest}
\frac{\int_I |fg|}{\|\ef{f}\|_{W_1^1(I)}}\le \frac{\int_I \left|(Df)(x)\int_c^x|g|\right|\,dx}{\|{\textstyle \frac{1}{w}} D\ef{f}\|_{L^1(I)}}\le \|\ef{g}\|_{\Cs{\infty}}, 
\end{equation}
that is $\ef{g}\in (W_1^1(I))'_\strong$ and  $\|\ef{g}\|_{(W_1^1(I))'_\strong}\le \|\ef{g}\|_{\Cs{\infty}}$.

Now let $\ef{g}\in (W_1^1(I))'_\strong$.  Since (see \cite[Lemma 2.8]{BenSharp-1988}) for  $\ef{g}\in L^0(I)$ there are the equalities $\|\ef{g}\|_{(L^1(I))'_\strong}=\|\ef{g}\|_{(L^1(I))'_\weak}=\|\ef{g}\|_{L^\infty(I)}$, we obtain
\begin{align}
\|\ef{g}\|_{(W_1^1(I))'_\strong}&\ge \sup_{f\in\ef{f_{|h|}}\in W_1^1(I)\setminus\{0\}}\frac{\int_I |fg|}{\|\ef{f_{|h|}}\|_{W_1^1(I)}}\ge \sup_{\ef{h}\in L^1(I)\setminus\{0\}}\frac{\int_I |g(y)|\left[\int_y^d w|h|\right]\,dy}{2\|\ef{h}\|_{L^1(I)}}\nonumber\\
&=\sup_{\ef{h}\in L^1(I)\setminus\{0\}}\frac{\int_I |h(x)|w(x)\left[\int_c^x |g|\right]\,dx}{2\|\ef{h}\|_{L^1(I)}} =\frac{1}{2}\|\ef{g}\|_{\Cs{\infty}},\label{lowest}
\end{align}
and \eqref{Wstr} is proved.

By \cite[Theorem 2.5]{PSU0} the equality $(W_1^1(I))'_\weak=(W_1^1(I))'_\strong=\Cs{\infty}$ holds. Besides that, for any $g\in\ef{g}\in \Cs{\infty}$, $\ef{f}\in W_1^1(I)$, $f=\ef{f}\cap AC_\loc(I)$, we have
\begin{equation}\label{perestint}
\int_I fg=\int_I g(x)\left(\int_x^d Df\right)\,dx=\int_I (Df)(y)\left(\int_c^y g\right)\,dy. 
\end{equation}
Hence, similarly to \eqref{upperest} and \eqref{lowest}  we get $\|\ef{g}\|_{(W_1^1(I))'_\weak}\approx \|\ef{g}\|_{\Ch{\infty}}$, and \eqref{Wslab} is proved.

Further, for any $a\in I$ there exists the function $f\in X$ such that $\chi_{(c,a)}f=\chi_{(c,a)}$, and it implies $(\pio{X},\|\cdot\|_{W_1^1(I)})'_\weak\subset L^1_\loc([c,d))$. 
Consequently, for any  $a\in I$, $f\in X$, $g\in\ef{g}\in L^1_\loc([c,d))$ in view of decreasing $w$ we have
\begin{align*}
&\left|\int_a^d fg\right|=\left|\int_a^d g(x)\left(\int_x^d Df\right)\,dx\right|=\left|\int_a^d (Df)(y)\left(\int_a^y g\right)\,dy\right|\\
&\le \|\pio{f}\|_{W_1^1(I)}\sup_{y\in [a,d)}\left|w(y)\left[\int_c^y g-\int_c^a g\right]\right|\le 2\|\ef{g}\|_{\Ch{\infty}}\|\pio{f}\|_{W_1^1(I)}. 
\end{align*}
By passing to the limit as $a\to c+$, we get 
$\|\ef{g}\|_{(\pio{X},\|\cdot\|_{W_1^1(I)})'_\weak}\le 2\|\ef{g}\|_{\Ch{\infty}}$.

If $g\in\ef{g}\in (\pio{X},\|\cdot\|_{W_1^1(I)})'_\weak$ and $h\in\ef{h}\in L^1(I)$ with  $\supp\ef{h}\subset (c,b]$ for some $b\in I$, the equalities \eqref{perestint} hold with $f:=f_h$.
Therefore,
\begin{align*}
\|\ef{g}\|_{(\pio{X},\|\cdot\|_{W_1^1(I)})'_\weak}&\ge \sup_{b\in I}\sup_{\ef{h}\in L^1(I)\setminus\{0\},\supp\ef{h}\subset (c,b]}\frac{\left|\int_c^b g(y)\left[\int_y^b w h\right]\,dy\right|}{2\|\ef{h}\|_{L^1(I)}}\\
&\ge \sup_{b\in I}\sup_{\ef{h}\in L^1(I)\setminus\{0\},\supp\ef{h}\subset (c,b]}\frac{\left|\int_c^b h(x)w(x)\left[\int_c^x g\right]\,dx\right|}{2\|\ef{h}\|_{L^1(I)}}\\
&=\frac{1}{2}\sup_{b\in I}\sup_{x\in(c,b]}\left|w(x)\int_c^x g\right|=\frac{1}{2}\|\ef{g}\|_{\Ch{\infty}}, 
\end{align*}
and \eqref{WX} follows. 
\end{proof}

\section{Associated spaces with $\Ch{\infty}$}\label{main}

As in for the case $p<\infty$ the ``strong'' space associated with $\Ch{\infty}$ is the null space.
 
\begin{lem} \label{chasto} Let $w\in\eqref{genw}$, $[a,b]\subset I$ and $h\in\ef{h}\in L^1([a,b])$. For any $\varepsilon >0$ there exists $f\in\mathfrak{M}(I)$ such that $\supp f\subset [a,b]$, $|f|=|h|$ on $(a,b)$ and $\|\pio{f}\|_{\Ch{\infty}}<\varepsilon$. 
 \end{lem}

\begin{proof} Fix an arbitrary $\varepsilon>0$. Choose  $n>\frac{1}{\varepsilon}\left[\esup_{y\in [a,b]}w(y)\right]\int_a^b|h|$. Let $\{\alpha_i\}_{i=0}^{n}$ be partition of $[a,b]$ such that $\int_{\alpha_i}^{\alpha_{i+1}}|h|=\frac{1}{n}\int_a^b|h|$. For each $i\in\{0,\ldots,n-1\}$ we choose
 $\beta_i\in[\alpha_i,\alpha_{i+1}]$ with property $\int_{\alpha_i}^{\beta_i}|h|=\int_{\beta_i}^{\alpha_{i+1}}|h|$.
 Put
 $f:=|h|\sum_{i=0}^{n-1}(\chi_{[\alpha_i,\beta_i]}-\chi_{(\beta_i,\alpha_{i+1})})$. Then $|f|=|h|$ on $(a,b)$ and 
 \begin{align*}
\|\pio{f}\|_{\Ch{\infty}}&=\esup_{x\in (a,b]}w(x)\left|\int_a^x f\right|=\max_{0\le i\le n-1}\esup_{x\in (a_i,a_{i+1}]}w(x)\left|\int_{a_i}^x f\right|\\&\le\frac{1}{n}\int_a^b|h|\left[\esup_{y\in [a,b]}w(y)\right]<\varepsilon.  
 \end{align*} 
 \end{proof}
 
 \begin{cor}
  Let $w\in\eqref{genw}$ and $\ef{g}\in L^0(I)\setminus\{0\}$. Then there is the equality $\|\ef{g}\|_{(\Cs{\infty},\|\cdot\|_{\Ch{\infty}})'_\strong}=\infty$.
  \end{cor}

\begin{proof}
 Let $g\in\ef{g}\not=0$. There exists  $[a,b]\subset I$ such that $a<b$ and ${\cal L}^1((a,b)\cap \{x\in I:g(x)\not=0\})>0$. Fix an arbitrary $\varepsilon>0$.  By Lemma \ref{chasto} there exists $f\in \mathfrak{M}(I)$ with the properties: $\|\pio{f}\|_{\Ch{\infty}}<\varepsilon$, $\supp f\subset [a,b]$ and $|f|=1$ on $(a,b)$. Hence, 
 $$\frac{\int_{I}|fg|}{\|\pio{f}\|_{\Ch{\infty}}}\ge \frac{1}{\varepsilon}\int_a^b |g|.$$
 Thus, $\|\ef{g}\|_{(\Cs{\infty},\|\cdot\|_{\Ch{\infty}})'_\strong}=\infty$. 
 \end{proof}

The next two lemmas contain the key constructions for obtaining a criterion for an element to belong to the ``weak'' space associated with $\Ch{\infty}$. 

\begin{lem} \label{gevarlambdag} Let $w\in\eqref{weights}$.

$1$. If $\ef{g}\in (\Ch{\infty})'_\weak$ then $v\ef{g}\in L^1(I)$. If  $\ef{g}\in (\Cs{\infty},\|\cdot\|_{\Ch{\infty}})'_\weak$ and $v>0$ $\leb$-a.e. on $I$ then 
\begin{equation} \label{gegL1}
\|v\ef{g}\|_{L^1(I)}\le \|\ef{g}\|_{(\Cs{\infty},\|\cdot\|_{\Ch{\infty}})'_\weak}.
\end{equation}

$2$.  Let  $(a)$ $\ef{g}\in (\Ch{\infty})'_\weak$ and $A_{\ef{g}}:=\|\ef{g}\|_{(\Ch{\infty})'_\weak}$, or $(b)$ $v>0$ $\leb$-a.e. on $I$, $\ef{g}\in (\Cs{\infty},\|\cdot\|_{\Ch{\infty}})'_\weak$ and $A_{\ef{g}}:=\|\ef{g}\|_{(\Cs{\infty},\|\cdot\|_{\Ch{\infty}})'_\weak}$. Then
 $(\frac{1}{w}\ef{g})\cap BPV(I)\not=\emptyset$,  and  for any $\tilde g\in (\frac{1}{w}\ef{g})\cap BPV(I)$  the estimate
\begin{equation}\label{ocenkaslambda}
 \|\lambda_{\tilde g}\|\le \|v\ef{g}\|_{L^1(I)}+A_{\ef{g}}
\end{equation}
holds. 
\end{lem}

\begin{proof}
$1$. Since $\pio{v}\in \Ch{\infty}$ then $v\ef{g}\in L^1(I)$ for $\ef{g}\in (\Ch{\infty})'_\weak$.
Now let $v>0$ $\leb$-a.e. on $I$. Then for any  $f\in\ef{f}\in L^0(I)$ 
$$\|\ef{f}\|_{\Cs{\infty}}\le \|\pio{{\textstyle\frac{f}{v}}}\|_{L^\infty(I)},$$
and for $\ef{g}\in (\Cs{\infty},\|\cdot\|_{\Ch{\infty}})'_\weak$ the estimates 
\begin{align*}
\|\ef{g}\|_{(\Cs{\infty},\|\cdot\|_{\Ch{\infty}})'_\weak}&\ge\sup_{\ef{f}: \pio{\frac{f}{v}}\in L^\infty(I)}\frac{\left|\int_I \frac{f}{v}vg\right|}{\|\pio{\frac{f}{v}}\|_{L^\infty(I)}}=\sup_{h\in\ef{h}\in L^\infty(I)}\frac{\left|\int_I hvg\right|}{\|\ef{h}\|_{L^\infty(I)}}\\
&=\|v\ef{g}\|_{L^1(I)} 
\end{align*}
hold.

$2$. Fix an arbitrary $\phi\in C_c^1(I)$. Put 
 $f:=D(\frac{1}{w}\phi)=v\phi+\frac{1}{w} D\phi$. 
 Then  for $\ef{f}\ni f$  we have $\ef{f}\in L^1(I)$ and $\|\ef{f}\|_{\Ch{\infty}}=\max_{x\in I}|\phi(x)|$. If $v>0$ $\leb$-a.e. on $I$ then $\ef{f}\in \Cs{\infty}$. From $v\ef{g}\in L^1(I)$ we have  $\int_{I} |\phi vg|<\infty$ for $g\in\ef{g}$.  Hence, $\int_{I}|\frac{1}{w} g D\phi|<\infty$  and
 $$\frac{|\int_{I} \frac{1}{w} g D\phi|}{\max\limits_{x\in I}|\phi(x)|}\le \frac{|\int_{I} \phi vg|}{\max\limits_{x\in I}|\phi(x)|}+\frac{|\int_{I} f g|}{\|\ef{f}\|_{\Ch{\infty}}}\le\|v\ef{g}\|_{L^1(I)}+A_{\ef{g}}.$$
 
 For $\phi\in C_c^1(I)$ we put $\Lambda \phi:=\int_{I} \frac{1}{w} g D\phi$. By the Hahn\,--\,\,Banach theorem there exists an extension $\tilde \Lambda\in (C_0(I))^*$ of the functional $\Lambda$ with the estimate
 $$\|\tilde \Lambda\|_{(C_0(I))^*}\le \|v\ef{g}\|_{L^1(I)}+A_{\ef{g}}.$$

 By the Riesz theorem \cite[6.19]{Rudin} on the representation of a linear continuous functional on $C_0(I)$ there exists  a unique regular real Borel measure $\lambda$
such that $\|\lambda\|=\|\tilde \Lambda\|_{(C_0(I))^*}$ and $\tilde\Lambda \varphi=\int_{I}\varphi\,d\lambda$ for any $\varphi\in C_0(I)$.

Define $h_g(x):=\lambda(I\cap (-\infty,x])$, $x\in I$. Then $h_g\in BPV(I)$ and applying \cite[5.41]{Leoni}, we have
 $$\int_{I} \frac{1}{w} g D\phi=\tilde\Lambda \phi=\int_{I}\phi\,d\lambda=-\int_{I}h_g D\phi$$
 for any $\phi\in C_c^1(I)$.
 Hence,  $\frac{1}{w}g+h_g$ $\leb$-a.e. on $I$ 
coincides with a constant function. Therefore,  $(\frac{1}{w}\ef{g})\cap BPV(I)\not=\emptyset$, and $\lambda_{\tilde g}=\lambda_{h_g}=\lambda$ hold  for $\tilde g\in (\frac{1}{w}\ef{g})\cap BPV(I)$ (see \cite[5.14]{Leoni}). 
 \end{proof}

 \begin{lem}\label{intfgfinite} Let $w\in\eqref{weights}$,  $f\in\ef{f}\in \Ch{\infty}$, $g\in\ef{g}\in  L^\infty_\loc(I)$, $v\ef{g}\in L^1(I)$, $\int_{I}|fg|<\infty$, $\tilde g\in (\frac{1}{w}\ef{g})\cap BPV(I)\not=\emptyset$. Then
 \begin{equation}\label{ocenka}
\left|\int_{I} fg\right|\le 2\left(\|v\ef{g}\|_{L^1(I)}+\|\lambda_{\tilde g}\|\right)\|\ef{f}\|_{\Ch{\infty}}.  
 \end{equation}
 \end{lem}

\begin{proof} Fix $\gamma\in (0,1)$. For $n\in\Na$ we define
$$b_n:=\sup\Big\{x\in I:\frac{1}{w(x)}\le n\Big\},\ \ a_n:=\inf\Big\{x\in I:\frac{1}{w(x)}\ge \frac{1}{n}\Big\}.$$
Since $\frac{1}{w}$ is a continuous function, $\lim_{x\to d-}\frac{1}{w(x)}=\infty$ and $\lim_{x\to c+}\frac{1}{w(x)}=0$, then 
$\frac{1}{w(b_n)}=n$, $\frac{1}{w(a_n)}=\frac{1}{n}$.
Besides that, since $\{x\in I:\frac{1}{w(x)}\le n\}\subset \{x\in I:\frac{1}{w(x)}\le n+1\}$, then $b_n\le b_{n+1}$. If $b:=\lim_{n\to\infty}b_n<d$ then $\pio{v}\in L^1_\loc([c,d))$ implies $\infty=\lim_{n\to\infty} \frac{1}{w(b_n)}=\frac{1}{w(b)}<\infty$ and we get contradiction. Hence, $\lim_{n\to\infty}b_n=d$.
 Analogously, $a_n\downarrow c$ as $n\to\infty$.

 Let $n_0\in\Na$ be such that $a_{n_0}< b_{n_0}$.
For  $n\ge n_0$ we define $\alpha_n\in [a_n,b_n]$  such that $\frac{1}{w(\alpha_n)}=\min_{x\in[a_n,b_n]} \frac{1}{w(x)}$. Then  $\frac{1}{w(\alpha_n)}>0$ and $\alpha_n<b_n$. Show that  $\lim_{n\to\infty}\alpha_n=c$. Fix an arbitrary $a>c$. Since $a_n\downarrow c$ as $n\to\infty$ there exists $n_1>n_0$ such that $a_{n_1}<a$. Let $n_2>n_1$ and $\frac{1}{n_2}<\frac{1}{w(\alpha_{n_1})}$. Then for  $n>n_2$ we have $\alpha_n\in [a_n,a_{n_1}]$ because of  $\frac{1}{w(x)}\ge n_1>\frac{1}{w(a_{n_1})}$ for $x\ge b_{n_1}$. Hence, $\alpha_n<a$.

Since $\int_{\alpha_n}^{b_n}|f|<\infty$ then by \cite[3.14]{Rudin} for $n\ge n_0$ there exists a function $\bar f_n\in C_c((\alpha_n,b_n))$ such that $\int_{\alpha_n}^{b_n}|f-\bar f_n|\le \frac{1}{w(\alpha_n)n}(1+\|\pio{g\chi_{[\alpha_n,b_n]}}\|_{L^\infty(I)})^{-1}$.
 Let us choose $\beta_n\in (b_n,d)$, $\theta_n\in\{-1,1\}$ so that the equality $\theta_n\int_{b_n}^{\beta_n}w^\gamma v+\int_{\alpha_n}^{b_n} \bar f_n=0$ holds.  For 
 $$f_n:=\bar f_n\chi_{[\alpha_n,b_n]}+\theta_n w^\gamma v\chi_{[b_n,\beta_n]},\  n\ge n_0$$
 we have
 $$\int_c^x f_n=0,\ \ x\in (c,\alpha_n]\cup [\beta_n,d),$$
\begin{align*}
&\sup_{x\in (\alpha_n,b_n]}w(x)\left|\int_c^x f_n\right|=\sup_{x\in (\alpha_n,b_n]}w(x)\left|\int_{\alpha_n}^x (\bar f_n-f)+\int_c^x f-\int_c^{\alpha_n} f \right|\\
&\le \frac{1}{n}+2\sup_{x\in [\alpha_n,b_n]}w(x)\left|\int_c^x f\right|\le 2\|\ef{f}\|_{\Ch{\infty}}+\frac{1}{n}
\end{align*}
and 
\begin{align*}
&\sup_{x\in (b_n,\beta_n)}w(x)\left|\int_c^x f_n\right|\le\\
&\le\sup_{x\in (b_n,\beta_n)}w(x)\left[\left|\int_{\alpha_n}^{b_n} (\bar f_n-f)+\int_c^{b_n} f-\int_c^{\alpha_n} f\right|+\int_{b_n}^x w^\gamma v\right]\\
&\le 2\|\ef{f}\|_{\Ch{\infty}}+\frac{1}{n}+ \sup_{x\in (b_n,\beta_n)}\frac{w(x)[w(x)^{\gamma-1}-w(b_n)^{\gamma-1}]}{(1-\gamma)}\\
&\le 2\|\ef{f}\|_{\Ch{\infty}}+\frac{1}{n}+\frac{1}{(1-\gamma)n^\gamma}.
\end{align*}

Besides that,  
\begin{align*}
\left|\int_{I}fg-\int_{I}f_ng\right|&\le \int_c^{\alpha_n} |fg|+\|\pio{g\chi_{[\alpha_n,b_n]}}\|_{L^\infty(I)}\int_{\alpha_n}^{b_n}|f-\bar f_n|+\\
&\phantom{\hskip 18mm}+\int_{b_n}^\infty|fg|+\left|\int_{\beta_n}^\infty w^\gamma vg\right|\\
&\le \int_c^{\alpha_n} |fg|+\frac{1}{n^2}+\int_{b_n}^\infty|fg|+\frac{1}{n^\gamma}\|v\ef{g}\|_{L^1(I)}. 
\end{align*}
Thus, $\lim_{n\to\infty}\int_{I} f_n g=\int_{I} fg$. 

Now put $F_n(x):=w(x)\int_c^x f_n$, $x\in I$. Then $F_n\in AC_\loc(I)$, $\supp F_n$ is a compact in $I$ and $f_n=vF_n+\frac{1}{w} DF_n$ $\leb$-a.e. on $I$. By applying \cite[5.40]{Leoni}, we obtain 
$$\int_{I} f_ng=\int_{I} vgF_n +\int_{I} \frac{1}{w} g DF_n=\int_{I} vg F_n -\int_{I} F_n\,d\lambda_{\tilde g}.$$
Consequently, 
$$\left|\int_{I} f_n g\right|\le \left(\|v\ef{g}\|_{L^1(I)}+\|\lambda_{\tilde g}\|\right)\sup_{x\in I}w(x)\left|\int_c^x f_n\right| $$
and  \eqref{ocenka} follows by passing to the limit as $n\to\infty$. 
\end{proof}

 Now we can formulate criterion of belonging of an element $\ef{g}\in L^0(I)$ to the space $(\Ch{\infty})'_\weak$ and get two sided estimate on norm of the element  of the ``weak'' space in case $v>0$ $\leb$-a.e. on $I$.
 
 \begin{thm}
  Let $w\in\eqref{weights}$, $\ef{g}\in L^0(I)$. The following statements are equivalent:
 
  $(i)$ $\ef{g}\in (\Ch{\infty})'_\weak$;
  
  $(ii)$  $v\ef{g}\in L^1(I)$, $(\frac{1}{w}\ef{g})\cap BPV(I)\not=\emptyset$, $\ef{g}\in L^\infty(I)$, and $\chi_{(b,d)}\ef{g}=0$ for some $b\in I$.
 
  And if $v>0$ $\leb$-a.e. on $I$, then  $\|\ef{g}\|_{(\Ch{\infty})'_\weak}\approx \left(\|v\ef{g}\|_{L^1(I)}+\|\lambda_{\tilde g}\|\right)$, where $\tilde g\in (\frac{1}{w}\ef{g})\cap BPV(I)$.
\end{thm}

 \begin{proof} $(ii)\Rightarrow (i)$. For $f\in\ef{f}\in \Ch{\infty}$ we have $\ef{f}\in L^1_{\loc}([c,d))$ and therefore
  $$\int_{I}|fg|\le\|\ef{g}\|_{L^\infty(I)}\int_c^b|f|<\infty.$$
  By applying Lemma \ref{intfgfinite} for $\tilde g\in (\frac{1}{w}\ef{g})\cap BPV(I)$ we get the estimate   $$\|\ef{g}\|_{(\Ch{\infty})'_\weak}\le 2\left(\|v\ef{g}\|_{L^1(I)}+\|\lambda_{\tilde g}\|\right)<\infty.$$

 $(i)\Rightarrow (ii)$. Denote $E:=\{x\in I: g(x)\not=0\}$. Assume that ${\cal L}^1((t,d)\cap E)>0$ for any $t\in I$.   Then there exists $\{[a_k,b_k]\}_1^\infty$ such that $b_k<a_{k+1}$ and $\int_{a_k}^{b_k}|g|>0$. Let us choose $\theta_k\in (0,\infty)$ so that the inequality $\theta_k\int_{a_k}^{b_k}|g|\ge 1$ holds. 
 By Lemma~\ref{chasto} there exists $f_k\in \mathfrak{M}(I)$ with properties: $\|\pio{f_k}\|_{\Ch{\infty}}<2^{-k}$, $\supp f_k\subset [a_k,b_k]$ and $|f_k|=\theta_k$ on $(a_k,b_k)$. 
 Then for the function
 $f:=\sum_{k=1}^\infty f_k$ we have $\|\pio{f}\|_{\Ch{\infty}}\le 1$
 and
 $$\int_{I}|fg|\ge \sum_{k=1}^\infty \theta_k \int_{a_k}^{b_k}|g|\ge\sum_{k=1}^\infty 1=\infty.$$
 This contradicts  $\ef{g}\in (\Ch{\infty})'_\weak$.
 Thus, there exists point $b\in I$ such that $\pio{g\chi_{(b,d)}}=0$.

 Assume that $\ef{g}\not\in L^\infty(I)$. Then there exists $h\in\ef{h}\in L^1((c,b))$ such that $\int_c^b |hg|=\infty$. Let $a_1:=b$ and $a_k\downarrow c$ as $k\to\infty$.
 By Lemma~\ref{chasto} there exists $f_k\in \mathfrak{M}(I)$ with properties: $\supp f_k\subset [a_{k+1},a_k]$, $\|\pio{f_k}\|_{\Ch{\infty}}<2^{-k}$ and $|f_k|=|h|$ on $(a_{k+1},a_k)$.
 Then for the function  
 $f:=\sum_{k=1}^\infty f_k$ we have $\|\pio{f}\|_{\Ch{\infty}}\le 1$ and
$$\int_{I}|fg|\ge \int_c^b|hg|=\infty.$$
This contradicts the relation $\ef{g}\in (\Ch{\infty})'_\weak$, that is $\ef{g}\in L^\infty(I)$.

By Lemma \ref{gevarlambdag}  we have $v\ef{g}\in L^1(I)$,  $(\frac{1}{w}\ef{g})\cap BPV(I)\not=\emptyset$. 
If $v>0$ $\leb$-a.e. on $I$ the statement $1$ of Lemma \ref{gevarlambdag} implies the estimate $3\|\ef{g}\|_{(\Ch{\infty})'_\weak}\ge \left(\|v\ef{g}\|_{L^1(I)}+\|\lambda_{\tilde g}\|\right)$ for $\tilde g\in (\frac{1}{w}\ef{g})\cap BPV(I)$. 
 \end{proof}

By using the results for the weighted Ces\`aro space,  we also can characterize the space $(\Cs{\infty},\|\cdot\|_{\Ch{\infty}})'_\weak$ in case $v>0$ $\leb$-a.e. on $I$. 
 
\begin{thm}
  Let $w\in\eqref{weights}$,  $v>0$ $\leb$-a.e. on $I$ and $\ef{g}\in L^0(I)$. The following statements are equivalent:
 
   $(i)$ $\ef{g}\in (\Cs{\infty},\|\cdot\|_{\Ch{\infty}})'_\weak$;
   
  $(ii)$ $v\ef{g}\in L^1(I)$,  $(\frac{1}{w}\ef{g})\cap BPV(I)\not=\emptyset$, and $\int_{I}v(t)\|\ef{g}\|_{L^\infty([t,d))}\,dt<\infty$.
 
 And   $\|\ef{g}\|_{(\Cs{\infty},\|\cdot\|_{\Ch{\infty}})'_\weak}\approx \left(\|v\ef{g}\|_{L^1(I)}+\|\lambda_{\tilde g}\|\right)$, where $\tilde g\in (\frac{1}{w}\ef{g})\cap BPV(I)$.
\end{thm}  
 
\begin{proof} At first,  $\int_{I}v(t)\|\ef{g}\|_{L^\infty([t,d))}\,dt<\infty$ is equivalent to  $\ef{g}\in (\Cs{\infty})'_\weak$ by \cite[4.3]{AM}, \cite[Theorem 4]{StepMZ2022}. 

$(ii)\Rightarrow (i)$. Since $\int_{I}v(t)\|\ef{g}\|_{L^\infty([t,d))}\,dt<\infty$ then  $\ef{g}\in L^\infty_\loc(I)$. Besides that, for $f\in\ef{f}\in \Cs{\infty}$ we have $\int_{I}|fg|<\infty$, and the estimate   
$$\|\ef{g}\|_{(\Cs{\infty},\|\cdot\|_{\Ch{\infty}})'_\weak}\le 2\left(\|v\ef{g}\|_{L^1(I)}+\|\lambda_{\tilde g}\|\right)<\infty.$$
 follows from Lemma \ref{intfgfinite}.
 
 $(i)\Rightarrow (ii)$. By Lemma \ref{gevarlambdag} we have $v\ef{g}\in L^1(I)$, $(\frac{1}{w}\ef{g})\cap BPV(I)\not=\emptyset$ and the estimate $3\|\ef{g}\|_{(\Cs{\infty},\|\cdot\|_{\Ch{\infty}})'_\weak}\ge \left(\|v\ef{g}\|_{L^1(I)}+\|\lambda_{\tilde g}\|\right)$ holds for  $\tilde g\in (\frac{1}{w}\ef{g})\cap BPV(I)$. Further, since $\|\ef{g}\|_{(\Cs{\infty})'_\weak}\le \|\ef{g}\|_{(\Cs{\infty},\|\cdot\|_{\Ch{\infty}})'_\weak}$, we obtain $\int_{I}v(t)\|\ef{g}\|_{L^\infty([t,d))}\,dt<\infty$. 
 \end{proof}

\end{document}